\documentclass[12pt,twoside,a4paper]{article}

\usepackage{a4,enumerate}
\usepackage[noadjust]{cite}
\usepackage{amssymb,amsmath,amsthm}
\usepackage{comment}
\usepackage{calc}
\usepackage{xcolor}

  \newcommand\mytitle{Increasing sequences of sectorial forms} 
\newcommand\lhead{H. Vogt, J. Voigt}
\newcommand\rhead{Increasing sequences of sectorial forms}
\swapnumbers
\numberwithin{equation}{section}

\newtheorem{theorem}{Theorem}[section]

\newtheorem{proposition}[theorem]{Proposition}
\newtheorem{lemma}[theorem]{Lemma}
\theoremstyle{definition}

\newtheorem{remark}[theorem]{Remark}

\newtheorem{example}[theorem]{Example}

 \mathchardef\ordinarycolon\mathcode`\:
  \mathcode`\:=\string"8000
  \begingroup \catcode`\:=\active
    \gdef:{\mathrel{\mathop\ordinarycolon}}
  \endgroup

\newcommand\rlim{
\mathchoice{\vcenter{\hbox{${\scriptstyle{+}}$}}}
{\vcenter{\hbox{$\scriptstyle{+}$}}}
{\vcenter{\hbox{$\scriptscriptstyle{+}$}}}
{\vcenter{\hbox{$\scriptscriptstyle{+}$}}}}

\newcommand\smid{\nonscript \mskip2mu plus2mu {\mid}%
\nonscript \mskip2mu plus2mu}     
\newcommand\scpr[2]{{(#1\smid#2)}}

\newcommand\dom{\operatorname{dom}}
\newcommand\ran{\operatorname{ran}}
\newcommand\Arg{\operatorname{Arg}}

\renewcommand{\Re}{\operatorname{Re}}
\renewcommand{\Im}{\operatorname{Im}}
\newcommand\imu{{\rm i}}
\renewcommand\phi{\varphi}

\renewcommand\epsilon{\varepsilon}

\newcommand{\R}{\mathbb{R}\nonscript\hskip.03em}
\newcommand{\N}{\mathbb{N}\nonscript\hskip.03em}
\newcommand{\C}{\mathbb{C}\nonscript\hskip.03em}
\newcommand{\K}{\mathbb{K}\nonscript\hskip.03em}
\newcommand\cL{\mathcal L}
\newcommand\la{\lambda}

\newcommand\tmo{^{-1}}
\newcommand\toh{^{1/2}}      
\newcommand\tmoh{^{-1/2}}    
\newcommand\tint{{\textstyle\int}}

\newcommand\norm[1]{\|#1\|}

\newcommand\ndash{\rule[.58ex]{\widthof{--}}{0.065ex}} 

\newcommand\eul{{\rm e}}

\makeatletter
\let\qedhere@ams\qedhere
\def\qedhere{\@ifnextchar[{\@qedhere}{\qedhere@ams}}
\def\@qedhere[#1]{\tag*{\raisebox{-#1ex}{\qedhere@ams}}}
\makeatletter
\def\env@cases{%
  \let\@ifnextchar\new@ifnextchar
  \left\lbrace
  \def\arraystretch{1.1}%
  \array{@{\,}l@{\quad}l@{}}%
}
\renewcommand\section{\@startsection {section}{1}{\z@}%
                                     {-3.25ex \@plus -1ex \@minus -.2ex}%
                                     {1.5ex \@plus.2ex}%
                                     {\normalfont\large\bfseries}}
\makeatother

\newcommand\set[2]{\bigl\{#1{;}\penalty300\;#2\bigr\}}
\newcommand\bset[2]{\Bigl\{#1{;}\;#2\Bigr\}}
\newcommand\ol{\overline}

\newcommand\smallbds{\vskip-1\lastskip\vskip5pt plus3pt minus2pt\noindent}
\newcommand\rfrac[2]{\tfrac{#1}{\raisebox{0.1em}{$\scriptstyle#2$}}}

\renewcommand\le{\leqslant}

\renewcommand\ge{\geqslant}

\newcommand\sse{\subseteq}
\newcommand\di{\mathclose{}\,\mathrm{d}}

\newcommand\slim{\mathop{\rm s\kern.08em\mbox{\rm -}lim}} 

\newcommand\abstracttext{\noindent
We prove convergence results for `increasing' sequences of 
sectorial forms. We treat both the case of closed forms and the case of
non-closable forms.
\vspace{8pt}

\noindent
MSC 2010: 47A07
\vspace{2pt}

\noindent
Keywords: sectorial form, strong resolvent convergence
}
\begin{document}
\title{\mytitle}

\author{Hendrik Vogt and J\"urgen Voigt}

\date{}

\maketitle

\begin{abstract}
\abstracttext
\end{abstract}

\section*{Introduction}
\label{intro}

In this paper we treat a convergence theorem for 
increasing sequences of sectorial forms in a complex Hilbert space.
More precisely, we will deal with 
a sequence $(a_n)$ of sectorial forms with vertex~$0$ and angle 
$\theta\in[0,\pi/2)$, with $\dom(a_n)\supseteq\dom(a_{n+1})$ and such 
that $a_{n+1}-a_n$ is sectorial with vertex~$0$ and angle $\theta$ for 
all $n\in\N$. (This is what we mean by `increasing sequence'. The setup 
implies that the real parts of the forms constitute an increasing sequence of 
symmetric forms, in the usual sense.) We do not assume 
that the forms are densely defined, and hence one does not obtain an operator 
$A_n$ associated with $a_n$, but rather a linear relation. These notions will 
be explained in more detail in Section~\ref{sec-prelims} of the paper.
The aim is to show that the linear relations converge to a limit linear
relation $A$ in strong resolvent sense; see Theorem~\ref{thm-main} for the case
of closed forms and Theorem~\ref{thm-monconv-ncl} for the case of non-closable
forms.

The history of this kind of convergence results starts with the treatment of 
increasing 
sequences of closed accretive \emph{symmetric} forms, by 
Kato~\cite[Theorem~VIII.3.13]{Kato1966} and 
Simon~\cite[Theorem~3.1]{Simon1978}; see also 
Kato~\cite[Theorem~VIII.3.13a]{Kato1980}. It 
was Simon~\cite[Section~4]{Simon1978} who advocated the use of non-densely 
defined (closed accretive symmetric) forms; the notions of non-densely
defined \emph{sectorial} forms and
their associated linear relations were developed
in~\cite{Kunze2005} and~\cite{BattyElst2014}. The 
result for sectorial forms sketched above, for the case of closed forms, is due 
to Batty and ter 
Elst~\cite[Theorem~2.2]{BattyElst2014}; there it is 
formulated in a 
different guise, with series of sectorial forms. 
Previously, a related kind of sequences of 
sectorial forms had been treated by Ouhabaz~\cite[Theorem~5]{Ouhabaz1995}. 

Our proof of Theorem~\ref{thm-main} 
is inspired by \cite[proof of Theorem~5]{Ouhabaz1995}. 
Whereas in~\cite[proof of Theorem~2.2]{BattyElst2014} the convergence in  
strong resolvent sense is proved directly for the sequence under consideration, 
we use the existing convergence result for the case of increasing sequences of 
symmetric forms and then use Kato's holomorphic families of type~(a) together 
with Vitali's convergence theorem. In 
Remark~\ref{rem-ouhabaz} we show that Ouhabaz' result 
\cite[Theorem~5]{Ouhabaz1995} can be obtained as a corollary of 
Theorem~\ref{thm-main}.

Finally, we also treat the case where the sequence consists of
non-closable forms.
Again, the result we prove is due to Batty and ter Elst 
\cite[Theorem~3.2]{BattyElst2014}. It is remarkable that Ouhabaz' procedure can 
also be adapted to this case and yields a proof that is barely more 
complicated than the proof for the case of closed forms.
\medskip

In Section~\ref{sec-prelims} we explain our notation concerning linear 
relations 
associated with non-densely defined sectorial forms.

In Section~\ref{sec-sf-salr} we discuss the correspondence between the order of 
non-densely defined closed accretive symmetric forms and the inverses of the 
corresponding linear relations. Our treatment is 
motivated by~\cite[Lemma~VI.2.30]{Kato1980}.

Section~\ref{section-monconv} is devoted to the main result for the case of 
closed forms.

In Section~\ref{section-monconv-non-cl} we treat the case of non-closable forms.

\section{Preliminaries on sectorial forms, linear relations and degenerate 
semigroups}
\label{sec-prelims}

Let $H$ be a complex Hilbert space. A sectorial form $a$ in $H$ with vertex~$0$ 
and angle $\theta\in[0,\pi/2)$
is a sesquilinear map $a\colon \dom(a)\times\dom(a)\to\C$,
where the domain $\dom(a)$ is a subspace of $H$ and
\[
a(u):= a(u,u)\in \ol{\Sigma_\theta}\qquad (u\in\dom(a)),
\]
with $\Sigma_\theta:=\set{z\in\C\setminus\{0\}}{\mathopen|\Arg z|<\theta}$ 
if $0<\theta<\pi/2$, 
and $\Sigma_0:=(0,\infty)$. We define
\[
a^*(u,v):=\ol{a(v,u)}\qquad (u,v\in\dom(a^*):=\dom(a)),
\]
$\Re a:=\frac12(a+a^*)$, $\Im a:=\frac1{2\imu}(a-a^*)$ and
\[
\norm u_a:=(\Re a(u) +\norm u_H^2)\toh\qquad (u\in\dom(a)).
\]
The form $a$ is called closed if $(\dom(a),\norm\cdot_a)$ is complete.

Let $a$ be a closed sectorial form. Then an m-sectorial operator $A_0$ in 
$H_0:=\ol{\dom(a)}$ is associated with $a$ via
\[
A_0:=\set{(u,y)\in\dom(a)\times H_0}{a(u,v)=\scpr yv_{H_0}\ (v\in\dom(a))},
\]
and $A_0$ is extended to an m-sectorial linear relation in $H$ by
\begin{align}\label{eq-form-relation}
\begin{split}
A &:=A_0\oplus(\{0\}\times H_0^\perp)\\
&\phantom:=\set{(u,y)\in\dom(a)\times H}{a(u,v)=\scpr yv_{H}\ (v\in\dom(a))}.
\end{split}
\end{align}
The m-sectoriality of $A$, with vertex~$0$ and angle $\theta$, means that
\[
\scpr yx\in\ol{\Sigma_\theta}\qquad((x,y)\in A)
\]
\smallbds
as well as
\[
\ran(I+A)=\set{x+y}{(x,y)\in A}=H.
\]
We point out
that each m-sectorial linear relation $A$ in $H$ is of the form 
\[
A=A_0\oplus(\{0\}\times H_0^\perp),
\]
where $H_0:=\ol{\dom(A)}$, and where $A_0:=A\cap(H_0\times H_0)$ is an 
m-sectorial operator in~$H_0$;
see 
\cite[first paragraph of Section~2]{BattyElst2014}.

Now assume additionally that the form $a$ is symmetric, i.e., 
$a^*=a$ (or equivalently, $a$ is 
sectorial with vertex~$0$ and angle $0$). Expressed differently, we now assume 
that $a$ is a closed accretive symmetric form. 
Then the operator $A_0$ described above is a self-adjoint operator in $H_0$, 
and $A$ is a self-adjoint linear relation, i.e.,
\[
A^*:=\bigl((-A)^\perp\bigr)\tmo = A
\]
(where the orthogonal complement of the linear relation 
$-A=\set{(x,-y)}{(x,y)\in A}$ is taken in $H\oplus H$); 
see~\cite[Section 5]{Arens1961}. 

Closing this section we mention that an m-sectorial linear relation $A$ also 
gives rise to a bounded holomorphic degenerate strongly continuous semigroup.
More precisely, the operator $-A_0$ from above generates a bounded holomorphic 
$C_0$-semigoup $T_0$ on $H_0$. Let $P_0\in\cL(H)$ denote the orthogonal 
projection onto $H_0$. Then $T(t):= T_0(t)P_0$ ($t\ge0$) defines a holomorphic 
degenerate 
strongly continuous semigroup $T$ on $H$. (`Degenerate' refers to 
the circumstance that  $T(0)$ may be different from the identity $I$, and 
`strongly continuous' to the property $T(0)=\slim_{t\to0\rlim} T(t)$.) The 
`generator' of $T$ is the linear relation $-A$.

We note that, for $t>0$, the operator $T(t)$ can be obtained as a contour 
integral
\begin{equation}\label{eq-contour-int-hol}
T(t)=\frac1{2\pi\imu}\int_{\Gamma}\eul^{-t\la}(\la I-A)\tmo\di\la,
\end{equation}
with a suitable (unbounded) path $\Gamma$ in~$\C$. This is traditional if 
$H_0=H$.
In the above general case it then follows from $(\la I-A)\tmo=(\la I_0-A_0)\tmo 
P_0$ (where $I_0$ is the identity on $H_0$) and \eqref{eq-contour-int-hol} with 
$T_0$ and $A_0$ in place of $T$ and $A$.

\section{On the order for symmetric forms and self-adjoint linear relations}
\label{sec-sf-salr}

Let $H$ be a Hilbert space. Given two accretive symmetric 
forms $a$ and $b$ in $H$, we say that $a\le b$ if 
$\dom(a)\supseteq\dom(b)$ and $a(u)\le b(u)$ for all $u\in\dom(b)$. 
A crucial fact needed in the proof of the form convergence theorem in
the case of symmetric forms is the following result.

\begin{proposition}\label{prop-monot}
Let $a$ and $b$ be closed accretive symmetric forms in $H$. Let $A$ and $B$ be 
the 
associated self-adjoint linear relations, and assume that 
the inverses of $A$ and $B$ are operators belonging to
$\cL(H)$. Then $a\le b$ if and only 
if $B\tmo\le\nobreak A\tmo$.
\end{proposition}

In \cite[proof of Theorem~4.1]{Simon1978} (which deals with the generalisation 
to the non-densely defined case) it is stated that `the proofs 
of Section~3 require no changes'. Somehow, the author 
seems to have overlooked that 
\cite[Proposition~1.1]{Simon1978} -- Proposition~\ref{prop-monot} as above 
-- is only stated for densely defined forms. It 
is no surprise that the equivalence in Proposition~\ref{prop-monot} also 
holds for 
non-densely defined forms; in fact, proofs can be found
in~\cite[Proposition~2.7]{Kunze2005} and \cite[Lemmas~3.2 
and~3.3]{HassiSandoviciSnooWinkler2006}.
Our proof of this equivalence is quite different from those proofs.

The key observation of our treatment is 
Proposition~\ref{lem2-monot} below, to which the following  
elementary lemma is a preparation.

\begin{lemma}\label{lem1-monot}
Let $X$ be a normed space, $H$ a Hilbert space, $P\in\cL(X,H)$, $\eta\in X'$,
$c\ge0$ with the property that
\begin{equation}\label{eq1-monot}
|\eta x|\le c\|Px\|\qquad (x\in X).
\end{equation}
Then there exists
$z\in H$ such that $\|z\|\le c$ and $\eta x=\scpr{Px}z$ for all $x\in X$.
\end{lemma}

\begin{proof}
Define $\tilde\eta\colon\ran(P)\to\K$ by $\tilde\eta(Px):=\eta x$ for all
$x\in X$; note that \eqref{eq1-monot} implies that $\tilde\eta$ is well-defined
and continuous on $\ran(P)$, $\|\tilde\eta\|\le c$. The Riesz{\ndash}Fr\'echet
representation theorem implies that there exists $z\in\ol{\ran(P)}$ such that
$\|z\|\le c$ and $\eta x=\tilde\eta(Px)=\scpr{Px}z$ for all $x\in X$.
\end{proof}

Assuming that $G$ and $H$ are Hilbert spaces and that $C$ and $D$ are linear
relations in $G\times H$, we will say that $D$ \textbf{dominates} $C$ if
for all $(x,y)\in D$ there exists $z\in H$ such that $(x,z)\in C$
and
$\|z\|\le\|y\|$. If $C$ and $D$ are operators, this simply means that
$\dom(D) \sse \dom(C)$
and $\norm{Cx} \le \norm{Dx}$ for all $x \in \dom(D)$. 
The following fundamental property concerning this notion is a
more elaborate version of \cite[Lemma~VI.2.30]{Kato1980}.

\begin{proposition}\label{lem2-monot}
Let $G,H$ be Hilbert spaces, and let $C,D$ be closed linear relations in
$G\times H$.
Then $D$ dominates $C$ if and only if $C^\perp$ dominates $D^\perp$.
\end{proposition}

\begin{proof}
It clearly suffices to show `$\Rightarrow$'. Let $(x,y)\in C^\perp$.

Let $(f,g)\in D$. By hypothesis, there exists $h\in H$ such that $(f,h)\in C$
and $\|h\|\le\|g\|$. Then $(f,h)\perp (x,y)$, hence
\[
|\scpr {-f}x|=|\scpr hy|\le\norm h\norm y\le\norm g\norm y,
\]
and with $P\colon D\to H,\ (f,g)\mapsto g$ and $\eta\colon D\to\K,\
(f,g)\mapsto \scpr
{-f}x$ it follows that $|\eta(f,g)|\le\norm
y\norm{P(f,g)}$
(note that $h$ has dropped out of these properties). Now we can apply
Lemma~\ref{lem1-monot} to obtain $z\in H$ such that $\norm z\le\norm y$ and
$\scpr {-f}x =\scpr{P(f,g)}z =\scpr gz$ for all $(f,g)\in D$, i.e.,
$(x,z)\in D^\perp$ and $\norm z\le\norm y$.

Summarising, we have shown that $C^\perp$ dominates $D^\perp$.
\end{proof}

\begin{proof}[Proof of Proposition~\ref{prop-monot}]
Let $A_0$ be the accretive self-adjoint operator in $\ol{\dom(a)}$ 
associated with $a$, and denote by $P_a$ 
the orthogonal projection onto $\ol{\dom(a)}$. The (accretive 
self-adjoint) square root of $A\tmo$ will be denoted by $A\tmoh$. We point 
out that then $A_0\tmoh$, the square root of 
$A_0\tmo$, is the restriction of $A\tmoh$ to $\overline{\dom(a)}$, and we 
define 
\[
A\toh:=(A\tmoh)\tmo= \set{(x,y)\in H\times H}{(x,P_ay)\in A_0\toh}.
\]
The corresponding notation and properties will also be used for~$b$. 

It is evident that $a\le b$ if and only if 
$B_0\toh$ dominates $A_0\toh$. The latter, in turn, holds if and only if 
$B\toh$ dominates $A\toh$. (Clearly, $A_0\toh$ dominates $A\toh$, but also 
conversely: if $(x,y)\in A\toh$, then $(x,P_ay)\in A_0\toh$ and 
$\norm{P_ay}\le\norm y$. By the same token, $B_0\toh$ and $B\toh$ dominate each 
other. As it is easily seen that domination is transitive, one concludes the 
last assertion above.)

On the other hand,
$A\tmo\ge B\tmo$ if and only if $A\tmoh$ dominates $B\tmoh$.

Finally we observe that
\[
(A\tmoh)^\perp=\bigl((-A\tmoh)^*\bigr)\tmo=-(A\tmoh)\tmo=-A\toh
\]
by the self-adjointness of $A\tmoh$, and similarly $(B\tmoh)^\perp = -B\toh$. 
Now, applying 
Proposition~\ref{lem2-monot} we conclude that $B\toh$ dominates $A\toh$ if and 
only if $A\tmoh$ dominates $B\tmoh$.

This proves the asserted equivalence.
\end{proof}

\section{Monotone convergence of sectorial forms}
\label{section-monconv}

In this section we prove the main result concerning closed forms.
The basic idea is to use the well-established convergence result
for symmetric forms and to `propagate' it to the case of
sectorial forms by holomorphy.

\begin{theorem}\label{thm-main}
Let $H$ be a complex Hilbert space, and let $\theta\in[0,\pi/2)$. Let $(a_n)$ 
be a sequence of closed sectorial forms in $H$ with vertex~$0$ and angle 
$\theta$, and 
assume that $\dom(a_n)\supseteq\dom(a_{n+1})$ and that $a_{n+1}-a_n$ is 
sectorial with vertex~$0$ and angle~$\theta$, for all $n\in\N$. For 
$n\in\N$ let $A_n$ be the m-sectorial linear relation associated with~$a_n$. 
Define
\[
\dom(a):=\bset{u\in\bigcap_{n\in\N}\dom(a_n)}{\sup_{n\in\N}\Re a_n(u)<\infty}.
\]
Then for all $u,v\in\dom(a)$ the limit $a(u,v):=\lim_{n\to\infty}a_n(u,v)$ 
exists, 
the form $a$ thus defined is a closed sectorial form with vertex~$0$ and angle 
$\theta$, and $A_n\to A$ ($n\to \infty$) in strong resolvent sense, where 
$A$ is the m-sectorial linear relation associated with $a$.
\end{theorem}

\begin{proof}
(i) 
Similarly as in the well-known symmetric case one shows that
$a$ as defined above is a closed sectorial form (see 
\cite[Lemma~2.1]{BattyElst2014});
for completeness we include the argument.
The Cauchy\ndash Schwarz inequality implies that $\dom(a)$ 
is a vector space, and from the 
sectoriality of the forms $a_n-a_m$ ($n>m$) and the polarisation identity one 
concludes that $\lim_{n\to\infty}a_n(u,v)$ exists for all 
$u,v\in\dom(a)$. Clearly, 
$a$ thus defined is a sectorial form. For the closedness of $a$ we have to show 
that $(\dom(a),\norm\cdot_a)$ is complete. Let $(u_n)$ be a 
$\norm\cdot_a$-Cauchy sequence in $\dom(a)$. Then $u:=\lim u_n$ exists in $H$, 
and $(u_n)$ is a $\norm\cdot_{a_m}$-Cauchy sequence, hence $u_n\to u$ in 
$(\dom(a_m),\norm\cdot_{a_m})$ since $a_m$ is closed, for all $m\in\N$. For all 
$n \in\N$ one has 
\[
\sup_{m\in\N}\Re a_m(u-u_n)\le\sup_{m\in\N,\, k\ge n}\Re a_m(u_k-u_n)
=\sup_{k\ge n}\Re a(u_k-u_n) <\infty.
\]
This inequality implies $u\in\dom(a)$, and
\[
 \Re a(u-u_n) \le\sup_{k\ge n}\Re a(u_k-u_n)\to 0\qquad (n\to\infty)
\]
shows that $u_n\to u$ in the $\norm\cdot_a$-norm.

(ii) Assume that all the forms $a_n$ are symmetric. In this case one has 
$a_n\le a_{n+1}$ for all $n\in\N$, and the assertion 
follows from \cite[Theorem~4.1]{Simon1978}; recall that for this general case 
our Proposition~\ref{prop-monot} replaces \cite[Proposition~1.1]{Simon1978}.

(iii) In the general case there exists $C>0$ such that 
\begin{align*}
\mathopen|\Im a_n(u)|&\le 
C\Re a_n(u)\qquad (u\in\dom(a_n)),\\[.6ex plus .3ex minus .2ex] 
\mathopen|\Im(a_{n+1}(u)-a_n(u))|&\le C\Re(a_{n+1}(u)-a_n(u))\qquad 
(u\in\dom(a_{n+1})) 
\end{align*}
for all $n\in\N$, and 
\[
\mathopen|\Im a(u)|\le C\Re a(u)\qquad (u\in\dom(a)). 
\]
Then for 
$z\in\Omega:=\set{z\in\C}{\mathopen|\Re z|<1/C}$ the forms
\[
a_{n,z}:=\Re a_n+z\Im a_n \quad (n\in\N),\qquad a_z:=\Re a+z\Im a
\]
are closed sectorial forms in $H$. Indeed, for $z\in\Omega$, $u\in\dom(a_n)$ we 
have
\begin{gather*}
\Re a_{n,z}(u)=\Re a_n(u) + \Re z \Im a_n(u)
\ge(1-\mathopen|\Re z|\,C)\Re a_n(u) \ge 0, \\[0.5ex]
\mathopen|\Im a_{n,z}(u)|=\mathopen|\Im z\Im a_n(u)|
\le \mathopen|\Im z|\,C\Re a_n(u)\le
\frac{\mathopen|\Im z|\,C}{1-\mathopen|\Re z|\,C}\Re a_{n,z}(u),
\end{gather*}
and similarly for $a$ instead of $a_n$. For $z\in\Omega$, $n\in\N$ let 
$A_{n,z}$ be the m-sectorial linear relation associated with $a_{n,z}$, and let 
$A_z$ be the m-sectorial linear relation associated with $a_z$.

For $x\in(-1/C,1/C)$, $n\in \N$, $u\in\dom(a_{n+1})$ we have
\begin{equation*}
x\Im\bigl(a_n(u)-a_{n+1}(u)\bigr)
\le\frac1C\mathopen{\bigl|}\Im\bigl(a_{n+1}(u)-a_n(u)\bigr)\bigr|\le\Re\bigl(a_
{ n+1 } (u)-a_n(u)\bigr),
\end{equation*}
which implies
\[
a_{n,x}=\Re a_n+x\Im a_n \le \Re a_{n+1} + x\Im a_{n+1}= a_{n+1,x}.
\]
Hence, $(a_{n,x})_n$ is an increasing sequence of closed accretive symmetric 
forms.
Note that for all $x\in (-1/C,1/C)$ the limit form of the sequence
$(a_{n,x})$ has the domain
\[
\bset{u\in\bigcap_{n\in\N}\dom(a_{n,x})}{\sup_{n\in\N} 
a_{n,x}(u)<\infty}
         =\dom(a)=\dom(a_x),
\]
and that
$a_n(u)\to a(u)$ implies $a_{n,x}(u)\to a_x(u)$ as
$n\to\infty$, for all $u\in\dom(a)$.
From the case treated in part (ii) above we conclude that 
\begin{equation}\label{eq-conv-symm}
(I+A_{n,x})\tmo\to(I+A_x)\tmo \qquad(n\to\infty)
\end{equation}
strongly, for all 
$x\in(-1/C,1/C)$. 

For each $n\in\N$, the family $(a_{n,z})_{z\in\Omega}$ is a 
holomorphic family of type (a) in the sense of Kato, and similarly for the 
family $(a_z)_{z\in\Omega}$. By Kato \cite[Theorem~VII.4.2]{Kato1980} -- we 
also refer to the recent proof in \cite{VogtVoigt2018} -- this implies that the 
mappings
$\Omega\ni z\mapsto(I+A_{n,z})\tmo\in\cL(H)$ ($n\in\N$)
and $\Omega\ni z\mapsto(I+A_z)\tmo\in\cL(H)$ are holomorphic. 
(The quoted references only cover the case where the families
$(a_{n,z})_{z\in\Omega}$ and $(a_z)_{z\in\Omega}$ are defined on dense 
subspaces. For the general case
one has to use
the description \eqref{eq-form-relation} of the associated 
m-sectorial linear relations; see also Theorem~\ref{thm-hol-dep-forms}.)

In view of 
the convergence \eqref{eq-conv-symm} and the estimate 
$\|(I+A_{n,z})\tmo\|\le1$, for all 
$n\in\N$, $z\in\Omega$, Vitali's convergence theorem shows that 
$(I+A_{n,z})\tmo\to(I+A_z)\tmo$ ($n\to\infty$) strongly, for all $z\in\Omega$. 
(We refer to~\cite[Theorem~2.1]{ArendtNikolski2000} for an elegant proof of 
Vitali's theorem.)
In particular, setting $z=\imu$ we obtain $a_n=a_{n,\imu}$ for all $n\in\N$ and 
$a=a_\imu$, and therefore $(I+A_n)\tmo\to(I+A)\tmo$ ($n\to\infty$) strongly.
\end{proof}

\begin{remark}\label{rem-ouhabaz}
Ouhabaz \cite[Theorem~5]{Ouhabaz1995} proved the following convergence 
theorem. 
Let $a_n$ ($n\in\N$) and $a$ be densely defined closed sectorial forms with 
vertex~$0$ and angle $\theta_0\in[0,\pi/2)$, $(\Re a_n)_n$ increasing to $\Re a$
in the sense of Theorem~\ref{thm-main}, and suppose that $\Im 
a(u)=\lim_{n\to\infty}\Im 
a_n(u)$ for all $u\in\dom(a)$. Assume that $\Im a_n(u)\le\Im a_{n+1}(u)$ for 
all $u\in\dom(a_{n+1})$, $n\in\N$.
Then $(A_n)$ converges to $A$ in strong
resolvent 
sense, where $A_n$ is associated with $a_n$, for $n\in\N$, and $A$ is 
associated with $a$.

We show that this result can be obtained as a corollary of 
Theorem~\ref{thm-main} (or \cite[Theorem~2.2]{BattyElst2014}). 
Obviously,
\[
a_{n+1}(u)-a_n(u)\in\set{z\in\C}{0\le\Arg z\le\pi/2}.
\]
for all $u\in\dom(a_{n+1})$, $n\in\N$. Let 
$\eta\in(0,\pi/2-\theta_0)$. Then
replacing $a_n$ by 
$\eul^{-\imu\eta}a_n$ one easily checks that the sequence
$(\eul^{-\imu\eta}a_n)_n$ of
forms satisfyies the hypotheses of Theorem~\ref{thm-main} above, with
$\theta:=\max\{\theta_0+\eta,\pi/2-\eta\}$.
In order to check that the sequence $(\eul^{-\imu\eta}a_n)_n$ converges to the
form $\eul^{-\imu\eta}a$ we note that 
\begin{gather*}
\cos\theta_0\,|a_n(u)| \le \Re a_n(u) \le |a_n(u)|,\\[.6ex plus .3ex minus .2ex]
\cos\theta\,\bigl|\eul^{-\imu\eta}a_n(u)\bigr|\le
\Re\bigl(\eul^{-\imu\eta}a_n(u)\bigr)
\le\bigl|\eul^{-\imu\eta}a_n(u)\bigr|
\end{gather*}
for all $u\in\bigcap_{n\in\N}\dom(a_n)$, $n\in\N$,
and this implies that
\[
\bset{u\in\bigcap_{n\in\N}\dom(a_n)}{\sup_{n\in\N}\Re(\eul^{-\imu\eta} 
a_n)(u)<\infty}=\dom(a)= \dom(\eul^{-\imu\eta}a).
\]
Analogously one can treat the alternative case in 
\cite[Theorem~5]{Ouhabaz1995}
where the sequence $(\Im a_n)$ is decreasing instead of increasing, i.e.,
$\Im a_n(u)\ge\Im a_{n+1}(u)$ for all $u\in\dom(a_{n+1})$, $n\in\N$.

In view of the above, Theorem~\ref{thm-main} implies that one can relax the
hypotheses in Ouhabaz' result.
\end{remark}

The remainder of this section is devoted to explaining the implications of the 
convergence in Theorem~\ref{thm-main} for the associated degenerate strongly 
continuous semigroups. For all $n\in\N$ let $T_n$ be the holomorphic degenerate
strongly continuous semigroup generated by $-A_n$, and let $T$ be generated by 
$-A$. We are going to explain why 

(i) for all $u\in H$, $t>0$ one has 
\begin{equation}\label{eq-s-conv-sgps}
T(t)u =\lim_{n\to\infty}T_n(t)u,
\end{equation}
with uniform convergence on compact subsets of $(0,\infty)$,

(ii) for all  $u\in \ran(T(0))=\ol{\dom(a)}$, $t\ge0$ one has 
\eqref{eq-s-conv-sgps}, with uniform convergence on compact subsets of 
$[0,\infty)$.

A preliminary observation for the proof of both assertions is that
the set $\C\setminus \ol{\Sigma_\theta}$ is contained in
the resolvent sets of $A_n$ for all $n\in\N$ and of $A$,
by our hypotheses. Furthermore, fixing 
$\theta'\in(\theta,\pi/2)$, the resolvents obey an estimate
\begin{equation}\label{eq-res-estimate}
\norm{(\la-A_n)\tmo}\le\frac c{|\la|}
\end{equation}
on $\C\setminus(\Sigma_{\theta'}\cup\{0\})$, and the same for $A$, with a 
constant $c\ge0$ independent of $n$. Finally, the strong resolvent 
convergence implies that $(\la-A_n)\tmo\to(\la-A)\tmo$ strongly, uniformly for 
$\la$ in compact subsets of $\C\setminus(\Sigma_{\theta'}\cup\{0\})$.

In order to show the convergence in (i) we now specify the path $\Gamma$ 
mentioned in the formula \eqref{eq-contour-int-hol} as the boundary of the set 
$\Sigma_{\theta'}\cup B_\C(0,1)$, oriented `counterclockwise'. Then the 
assertion is an easy consequence of \eqref{eq-contour-int-hol} for the 
semigroups $T_n$ and~$T$, the estimate \eqref{eq-res-estimate} and the strong 
resolvent convergence 
formulated at the end of the previous paragraph; 
see~\cite[Theorem~2.3]{Kunze2005}. Another source for (i) 
is~\cite[Theorem~5.2]{Arendt2001} (where the erroneously asserted uniform 
convergence on $[0,\tau]$ has to be replaced by uniform convergence on compact 
subsets of $(0,\infty)$).

For the convergence (ii) we refer to \cite[Theorem~4.2(a)]{Arendt2001}.
Alternatively, one can prove (ii) by adapting the proof of the Trotter 
approximation theorem in \cite[Section~3.4]{Pazy1983} to the case of degenerate 
strongly continuous semigroups.

\section{The case of non-closable sectorial forms}
\label{section-monconv-non-cl}

In this section we show that the following result, contained in
\cite[Theorem~3.2]{BattyElst2014}, 
can be obtained by the method presented in 
Section~\ref{section-monconv}.

\begin{theorem}\label{thm-monconv-ncl}
Let $H$ be a complex Hilbert space, and let $\theta\in[0,\pi/2)$. Let $(a_n)$ 
be a sequence of sectorial forms in $H$ with vertex~$0$ and angle $\theta$, and 
assume that $\dom(a_n)\supseteq\dom(a_{n+1})$ and that $a_{n+1}-a_n$ is 
sectorial with vertex~$0$ and angle~$\theta$, for all $n\in\N$. For 
$n\in\N$ let $A_n$ be the m-sectorial linear relation associated with~$a_n$. 

Then there exists an m-sectorial linear relation $A$ such that $A_n\to A$ 
($n\to\infty$) in strong resolvent sense.
\end{theorem}

Note that the only difference between the hypotheses in this theorem and those 
in Theorem~\ref{thm-main} is that here the forms $a_n$ are no longer supposed 
to be closed; the ``price'' one has to pay is that in the conclusion
there is no longer a description of the limit linear relation in terms of
the forms $a_n$.

Before entering our proof of this result we have to explain how a linear 
relation is associated with a non-closable -- short for ``not necessarily 
closable'' -- form. Let 
$a$ be a 
sectorial form in $H$ with vertex~$0$ and angle $\theta\in[0,\pi/2)$. The norm 
$\|\cdot\|_a$ on $\dom(a)\sse H$ is defined as in Section~\ref{sec-prelims}. 
Let $V$ be the completion of $(\dom(a),\|\cdot\|_a)$. 
Then the 
continuous extension $j\colon V\to H$ of the embedding 
$\dom(a)\hookrightarrow H$ is not necessarily injective; 
in fact, $a$ is called 
closable if $j$ is injective. 

The form $a$ on $\dom(a)$ has a unique 
continuous extension $\tilde a\colon V\times V\to \C$, and this extension 
is $j$-elliptic and sectorial with vertex~$0$ and angle $\theta$. 
(By definition, $\tilde a$ being $j$-elliptic means that there exists 
$\omega\in \R$ such that the form $(u,v)\mapsto 
\tilde a(u,v)+\omega\scpr{ju}{jv}_H$ is coercive. In the present case, 
this condition holds with $\omega=1$.) The (m-sectorial) linear 
relation $A$ associated with $(a,j)$ is then given by
\[
A=\set{(x,y)\in H\times H}{\exists\,u\in V\colon ju=x,\ \tilde 
a(u,v)=\scpr y{jv}_H\ (v\in V)}.
\]
If $a$ is symmetric, then the linear relation $A$ is self-adjoint, and the 
(closed!) accretive form associated with $A$ is the closure of the regular part 
$a_{\mathrm 
r}$ of $a$, 
defined in \cite[Section~2]{Simon1978}. For these definitions and properties we 
refer to \cite[Section~3]{ArendtElst2012} and \cite[Section~3]{BattyElst2014}.

Next, we state a generalisation of the result on holomorphic families of 
type~(a), used in the proof of Theorem~\ref{thm-main}.

\begin{theorem}\label{thm-hol-dep-forms}
Let $V$ and $H$ be complex Hilbert spaces, $j\in\cL(V,H)$, let $\Omega\sse\C$ 
be open, and let $\theta\in[0,\pi/2)$. For each $z\in\Omega$ let $a_z\colon 
V\times V\to\C$ be a bounded $j$-elliptic sectorial form with vertex $0$ and 
angle $\theta$, and let $A_z$ denote the m-sectorial linear relation associated 
with $(a_z,j)$. Assume that for all $u,v\in V$ the function $\Omega\ni z\mapsto 
a_z(u,v)\in\C$ is holomorphic.

Then the mapping $\Omega\ni z\mapsto(I+A_z)\tmo\in\cL(H)$ is holomorphic.
\end{theorem}

If $\ran(j)$ is dense in $H$, then the proof can be given in the same way 
as in~\cite[proof of Theorem~1.1]{VogtVoigt2018}. Otherwise we define 
$H_0:=\ol{\ran(j)}$ and apply the previous case to the m-sectorial operators 
$A_{z,0}$ associated with $(a_z,j)$ in $H_0$. Then the description 
\eqref{eq-form-relation} of $A_z$ yields the assertion. Indeed, the resolvent 
$(I+A_z)\tmo$ decomposes according to the orthogonal sum $H_0\oplus H_0^\perp$, 
where the restriction to $H_0$ is $(I_{H_0}+A_{z,0})\tmo$ and the restriction to 
$H_0^\perp$ is the zero operator.

\begin{proof}[Proof of Theorem~\ref{thm-monconv-ncl}]
(i) First we treat the case where all the forms $a_n$ are symmetric. In this 
case we refer to \cite[Corollary~2.4]{Simon1978} for the property that one has 
$a_{n,\mathrm r}\le a_{n+1,\mathrm r}$ for all $n\in\N$, where $a_{n,\mathrm 
r}$ denotes the regular 
part of $a_n$. 
(In the quoted reference the 
property needed here is only formulated for densely defined forms. However, it 
is immediate that the proof of \cite[Theorem~2.2]{Simon1978}, which is the 
basis for the quoted property, works also for non-densely defined forms.) Then 
we conclude from \cite[Theorem~4.1]{Simon1978} that the sequence $(A_n)$ 
converges (to some~$A$) in strong resolvent sense.

(ii) For the general case we define
\[
a_{n,z}:=\Re a_n +z\Im a_n\qquad(n\in\N,\ z\in\C)
\]
and note that the arguments as in step (iii) of the proof of 
Theorem~\ref{thm-main} yield a constant $C>0$ such that
\[
\mathopen|\Im a_{n,z}(u)|
\le
\frac{\mathopen|\Im z|\,C}{1-\mathopen|\Re z|\,C}\Re a_{n,z}(u),
\]
for all $z\in\Omega:=\set{z\in\C}{\mathopen|\Re 
z|<1/C}$, $n\in\N$, $u\in\dom(a_n)$, and
\begin{equation}\label{eq-est-ncl-forms}
a_{n,x} \le a_{n+1,x}\qquad\bigl(x\in(-1/C,1/C),\ n\in\N\bigr).
\end{equation}
In particular, for $z\in\Omega$, $n\in\N$ the form $a_{n,z}$ is sectorial; 
let 
$A_{n,z}$ be the linear relation associated with $a_{n,z}$ (as explained 
above). From 
\eqref{eq-est-ncl-forms} and step (i) we conclude that, for $x\in(-1/C,1/C)$, 
there exists an operator $R_x\in\cL(H)$ such that $(I+A_{n,x})^{-1}\to R_x$ 
($n\to\infty$) strongly.

For $n\in\N$, let $V_n$ be the completion of $(\dom(a_n),\|\cdot\|_{a_n})$, 
$j_n\colon V_n\to H$ the continuous extension of the embedding 
$\dom(a_n)\hookrightarrow H$ as described above, let $\tilde a_n$ be the 
continuous extension of $a_n$ to $V_n\times V_n$, and define the 
$j_n$-elliptic sectorial forms $\tilde a_{n,z}$ by
\[
\tilde a_{n,z}:=\Re \tilde a_n + z\Im\tilde a_n.
\]
Then $\tilde a_{n,z}$ is the continuous extension of $a_{n,z}$,
so $A_{n,z}$ is associated with $\tilde a_{n,z}$.
The application of Theorem~\ref{thm-hol-dep-forms} shows that
$\Omega\ni z\mapsto(I+A_{n,z})^{-1}\in\cL(H)$ is 
holomorphic. As in step (iii) of the proof of Theorem~\ref{thm-main}, Vitali's 
convergence theorem implies the strong convergence of 
$\bigl((I+A_n)^{-1}=(I+A_{n,\imu})^{-1}\bigr)_{n\in\N}$ to an operator 
$R\in\cL(H)$. It follows from Lemma~\ref{lem-limit-sect}, proved subsequently, 
that the linear relation $A$ determined by $R=(I+A)^{-1}$
is m-sectorial.
\end{proof}

\begin{lemma}\label{lem-limit-sect}
Let $H$ be a complex Hilbert space, $(A_n)$ a sequence of m-sectorial relations 
in $H$ with vertex $0$ and angle $\theta\in[0,\pi/2)$, and assume that 
$R:=\slim(I+A_n)^{-1}$ exists. Then the linear relation $A$ 
determined by $R=(I+A)^{-1}$ is m-sectorial with vertex $0$ and angle $\theta$.
\end{lemma}

\begin{proof}
Note that $A=R^{-1}-I=\set{(y,x-y)}{(x,y)\in R}$. Clearly, $\ran(I+A) = \dom(R)
= H$. Let $(x,y)\in R$. With 
$y_n:=(I+A_n)^{-1}x$, i.e., $(y_n,x-y_n)\in A_n$, we have $\scpr{x-y_n}{y_n}\in 
\ol{\Sigma_\theta}$ and $y_n\to y$, hence $\scpr{x-y}{y}\in 
\ol{\Sigma_\theta}$. 
\end{proof}

\begin{remark}\label{rem-on-conv-thms}
We did not derive Theorem~\ref{thm-monconv-ncl} as a corollary of 
Theorem~\ref{thm-main}, but rather the proofs of the two results are ``the 
same''. However, we comment on the following decisive difference. In 
Theorem~\ref{thm-main} we had the limit form $a$ at our disposal, and
each of the 
linear relations $A_z$ in step (iii) of the proof was associated with the 
form $a_z=\Re a+z\Im a$, whereas in the proof of 
Theorem~\ref{thm-monconv-ncl}, the linear relation $A$ is obtained in 
an indirect way as a limit in strong resolvent sense. 

An idea to prove Theorem~\ref{thm-monconv-ncl} as a direct 
corollary of Theorem~\ref{thm-main} would be as follows. Denote by $\hat a_n$ 
the closed sectorial form with $\dom(\hat a_n)\sse H$ associated with $A_n$. As 
by hypothesis $a_{n+1}-a_n$ is sectorial (uniformly in $n$), it is 
tempting to think that $\hat a_{n+1}-\hat a_n$ might be sectorial (with the 
same vertex and angle). However, this idea fails dramatically, as we 
will illustrate by Example~\ref{example-sec-forms}(a).

A noteworthy feature concerning the operators $A_{n,z}$ associated with
$\Re a_n+z\Im a_n$, for the non-closable forms $a_n$ presented in  
Example~\ref{example-sec-forms}, is illustrated in part (b) of this example: 
none of the families $(A_{n,z})_{z\in\Omega}$ corresponds 
to a family $(b + z c)_{z\in\Omega}$, with closed symmetric forms $b$ and $c$. 
\end{remark}

\begin{example}\label{example-sec-forms}
Let $H=L_2(0,1)$. 

(a) We present a sequence $(a_n)$ of non-closable sectorial forms 
in $H$, with vertex $0$ and angle $\pi/4$, $a_{n+1}-a_n$ symmetric and 
accretive for all $n$, where the associated closed forms $\hat a_n$ are all 
symmetric, and the sequence $(\hat a_n)$ is strictly \emph{decreasing}.

For $n\in \N$ we define $\dom(a_n):=C[0,1]$,
\[
a_n(u,v):=\tint u\ol v+nu(0)\ol{v(0)}
     +\imu\bigl(u(0)\tint\ol v + (\tint u)\ol{v(0)}\bigr).
\]
Then
\begin{align*}
&(\Re a_n)(u,v)=\tint u\ol v+nu(0)\ol{v(0)},\\
&(\Im a_n)(u,v)=u(0)\tint\ol v + (\tint u)\ol{v(0)}
\qquad (u,v\in\dom(a_n)).
\end{align*}
Because of $\mathopen|\tint u|\le\norm u_2$ and $|u(0)|\norm 
u_2\le\frac12\bigl(\norm u_2^2+|u(0)|^2\bigr)$ we obtain
\[
 \mathopen|\Im a_n(u)|\le\Re a_n(u)\qquad (u\in\dom(a_n),\ n\in\N).
\]
Also, $a_{n+1}(u)-a_n(u)=|u(0)|^2\ge0$ for all $u\in\dom(a_{n+1})$, $n\in\N$.

For all $n\in\N$, the $a_n$-norm on $C[0,1]$ is equivalent to $\norm 
u:=\bigl(\norm u_2^2+|u(0)|^2\bigr)^{1/2}$. Therefore, $V := L_2(0,1)\oplus\C$ 
is a completion 
$(C[0,1],\norm\cdot_{a_n})$, and the 
continuous extension $j\colon V\to L_2(0,1)$ of the embedding 
$C[0,1]\hookrightarrow L_2(0,1)$ is given by 
\[
j(u,\alpha)=u\qquad((u,\alpha)\in L_2(0,1)\oplus\C).
\]

The continuous extension $\tilde a_n$ of $a_n$ to $V\times V$ is given by
\[
\tilde a_n\bigl((u,\alpha),(v,\beta)\bigr)=
\tint u\ol v+n\alpha\ol\beta +
      \imu\bigl(\alpha\tint\ol v +(\tint u)\ol\beta\bigr).
\]
For $u,f\in L_2(0,1)$ we have $(u,f)\in A_n$ if and only if
there exists $\alpha\in\C$ such that for all $(v,\beta)\in V$ one has
\begin{equation}\label{eq-form-for-an-0}
\tint u\ol v+n\alpha\ol\beta+\imu\bigl(\alpha\tint\ol v+(\tint u)\ol\beta\bigr)
=\tilde a_n\bigl((u,\alpha),(v,\beta)\bigr)=\scpr f{j(v,\beta)}=\tint f\ol v.
\end{equation}
This property is equivalent to $\alpha=-\rfrac\imu n \tint u$ and $f= 
u+(\rfrac1n\tint u)1$. This shows that the operator $A_n$ associated with the 
form $a_n$ is given by $A_nu=u+(\rfrac1n\tint u)1$ ($u\in L_2(0,1)$), and the 
associated closed form $\hat a_n$ with $\dom(\hat a_n)\sse L_2(0,1)$ is given by
\begin{equation*}
 \hat a_n(u,v)=\tint u\ol v+ \rfrac1n\tint u\tint\ol v\qquad (u,v\in L_2(0,1)).
\end{equation*}
Hence all the forms $\hat a_n$ are symmetric, $\hat a_{n+1}\le\hat 
a_n$ for all $n\in\N$, and $A_n\to I$ in $\cL(L_2(0,1))$.

(b) For the forms from part (a), using \eqref{eq-form-for-an-0} with $z$ 
in place of $\imu$, one can compute the operators 
$A_{n,z}$ associated with the form $a_{n,z}:=\Re a_n+ z\Im a_n$, for 
$\mathopen|\Re z|<1$. The result is
\[
A_{n,z}u = u-\rfrac{z^2}n(\tint u)1\qquad (u\in L_2(0,1)),
\]
with associated closed form
\[
\hat a_{n,z}(u,v)=\tint u\ol v- \rfrac{z^2}n\tint u\tint\ol v
\qquad (u,v\in L_2(0,1)).
\]
In contrast, the 
forms $(\hat a_n)_z:=\Re\hat a_n+z\Im\hat a_n=\hat a_n$ do
not depend on $z$ since $\hat a_n$ is symmetric.
The only points $z$ where the forms $(\hat a_n)_z$ 
and $\hat a_{n,z}$ coincide are $z=\pm\imu$.
\end{example}

{
}
\bigskip

\noindent
Hendrik Vogt\\
Fachbereich Mathematik\\
Universit\"at Bremen\\
Postfach 330 440\\
28359 Bremen, Germany\\
{\tt 
hendrik.vo\rlap{\textcolor{white}{hugo@egon}}gt@uni-\rlap{\textcolor{white}{%
hannover}}bremen.de}\\[3ex]
J\"urgen Voigt\\
Technische Universit\"at Dresden\\
Fakult\"at Mathematik\\
01062 Dresden, Germany\\
{\tt 
juer\rlap{\textcolor{white}{xxxxx}}gen.vo\rlap{\textcolor{white}{yyyyyyyyyy}}%
igt@tu-dr\rlap{\textcolor{white}{%
zzzzzzzzz}}esden.de}

\end{document}